\documentclass[12pt]{amsart}
\usepackage{amssymb}
\usepackage[all]{xy}

\newtheorem{thm}{Theorem}[section]
\newtheorem{lem}[thm]{Lemma}
\newtheorem{cor}[thm]{Corollary}
\newtheorem{prop}[thm]{Proposition}
\newtheorem{ex}[thm]{Example}

\newtheorem*{prob*}{Open problem}

\theoremstyle{definition}

\newtheorem{defi}[thm]{Definition}

\theoremstyle{remark}

\newtheorem{rem}[thm]{Remark}
\newtheorem*{rem*}{Remark}

\DeclareMathOperator{\rad}{rad}

\DeclareMathOperator{\Aff}{Aff}
\DeclareMathOperator{\Aut}{Aut}

\DeclareMathOperator{\nil}{nil}

\newcommand{\kringel}{\mathbin{\raise1pt\hbox{$\scriptstyle\circ$}}}
\newcommand{\pkt}{\mathbin{\raise0pt\hbox{$\scriptstyle\bullet$}}}

\newcommand{\C}{\mathbb{C}}

\newcommand{\R}{\mathbb{R}}

\newcommand{\ad}{\mathop{\rm ad}}

\newcommand{\End}{\mathop{\rm End}}
\newcommand{\Der}{\mathop{\rm Der}}

\newcommand{\La}{\mathfrak{a}}

\newcommand{\Lf}{\mathfrak{f}}
\newcommand{\Lg}{\mathfrak{g}}
\newcommand{\Lh}{\mathfrak{h}}

\newcommand{\Ll}{\mathfrak{l}}
\newcommand{\Ln}{\mathfrak{n}}
\newcommand{\Lm}{\mathfrak{m}}
\newcommand{\Lr}{\mathfrak{r}}
\newcommand{\Ls}{\mathfrak{s}}

\newcommand{\im}{\mathop{\rm im}}

\newcommand{\al}{\alpha}
\newcommand{\be}{\beta}
\newcommand{\ga}{\gamma}

\newcommand{\la}{\lambda}

\newcommand{\ra}{\rightarrow}

\renewcommand{\phi}{\varphi}

\begin{document}


\title[Post-Lie Algebras]{Affine actions on Lie groups and post-Lie algebra structures}

\author[D. Burde]{Dietrich Burde}
\author[K. Dekimpe]{Karel Dekimpe}
\author[K. Vercammen]{Kim Vercammen}
\address{Fakult\"at f\"ur Mathematik\\
Universit\"at Wien\\
  Nordbergstr. 15\\
  1090 Wien \\
  Austria}
\email{dietrich.burde@univie.ac.at}
\address{Katholieke Universiteit Leuven\\
Campus Kortrijk\\
8500 Kortrijk\\
Belgium}
\email{karel.dekimpe@kuleuven-kortrijk.be}
\email{kim.vercammen@kuleuven-kortrijk.be}

\date{\today}

\subjclass{Primary 17B30, 17D25}
\thanks{The first author was supported by the FWF, Projekt P21683. He thanks
the K.U. Leuven Campus Kortrijk for its hospitality and support.
The second author expresses his gratitude towards the Erwin Schr\"odinger International
Institute for Mathematical Physics. The third author is supported by a 
Ph.D. fellowship of the FWO and by the Research Fund K.U.Leuven.}

\begin{abstract}
We introduce post-Lie algebra structures on pairs of Lie algebras $(\Lg,\Ln)$ defined on
a fixed vector space $V$. Special cases are LR-structures and pre-Lie algebra structures
on Lie algebras. We show that post-Lie algebra structures naturally arise in the study 
of NIL-affine actions on nilpotent Lie groups. We obtain several results on the existence
of post-Lie algebra structures, in terms of the algebraic structure of the two Lie algebras
$\Lg$ and $\Ln$. 
One result is, for example, that if there exists a post-Lie algebra structure on
$(\Lg,\Ln)$, where $\Lg$ is nilpotent, then $\Ln$ must be solvable. 
Furthermore special cases and examples are given. This includes a classification of all
complex, two-dimensional post-Lie algebras.
\end{abstract}

\maketitle

\section{Introduction}
Post-Lie algebras have been introduced recently by Vallette in \cite{VAL}, in connection with 
homology of partition posets, and the study of Koszul operads. Moreover, they have been 
discussed in several articles of Loday, see for example \cite{LOD} and the references given therein.
Furthermore, post-Lie algebras also turned up in relation with the classical
Yang-Baxter equation in \cite{BGN}.
In this paper, we will show that these algebras also naturally appear in yet another context, 
namely in that of affine actions on Lie groups. 
Let $N$ be a real, connected and simply connected nilpotent Lie group. Then
the group $\Aff(N)=N\rtimes \Aut(N)$ acts by so called NIL-affine transformations on
$N$ via the rule $^{(m,\alpha)}n=m\cdot \alpha(n)$ for all $m,n\in N$ and 
for all $\alpha \in \Aut(N)$. For the special case where $N=\R^n$, we obtain the
usual group of affine transformations $\Aff(\R^n)=\R^n\rtimes GL_n(\R)$ acting by
affine transformations on $\R^n$. Suppose that $G$ is a connected and simply connected
solvable Lie group of dimension $n$. In the seventies J. Milnor posed a famous question in 
\cite{MIL}, whether or not any such $G$ would admit a representation $\rho\colon G\ra \Aff(\R^n)$,
letting $G$ act {\it simply transitively} on $\R^n$. This question received a lot of attention,
including several articles trying to prove a positive answer. However, finally it turned out
that the answer was negative \cite{BEN, BUG}. The main ingredient in this study was the notion 
of a pre-Lie algebra. On the other hand, it was shown \cite{BAU,DEK} 
that for any such $G$ there exists a nilpotent Lie group $N$ and an embedding $\rho\colon 
G \hookrightarrow \Aff(N)$,
letting $G$ act simply transitively on $N$. This result was, among other things, a motivation
to study pairs of Lie groups $(G,N)$ where $G$ acts simply transitively via NIL-affine 
transformations on $N$. In \cite{BDD},\cite{BU34} and \cite{BU38} we have obtained 
results for several cases. Just as in the usual affine setting, the translation of the problem 
to the Lie algebra level plays a crucial role. As we will explain at the end of section $2$, it 
turns out that in this more general setting, one is naturally led to study 
{\it post-Lie algebra structures} on a pair $(\Lg,\Ln)$ of Lie algebras. In particular, the 
notion of a post-Lie algebra appears again. 

Not only motivated by this geometric context, but also because of the usefulness
in several other mathematical fields,
it is our aim to start in this paper a more detailed algebraic study of post-Lie algebra
structures.

\section{Post-Lie algebra structures}

Pre-Lie algebras play an important role in many areas, 
in particular in geometry and physics, see \cite{BU24} for a survey.
A post-Lie algebra is a certain generalization of a pre-Lie algebra. It has been
defined by Valette in \cite{VAL}.
However, we give here the definition such that the associated pre-Lie algebra is a left pre-Lie algebra (also known as
left symmetric algebra), whereas Valette uses right pre-Lie algebras instead.

In this paper all algebras will be finite dimensional over a field $k$ of characteristic 0.
\begin{defi}\label{postlie}
A {\it post-Lie algebra} $(V,\cdot,\{\, ,\})$ is a vector space $V$ over a field $k$
equipped with two $k$-bilinear operations $x\cdot y$ and $\{x,y\}$, which satisfy the
relations
\begin{align}
0 & = \{x,y\} + \{y,x\} \label{post0}\\
 0 & = \{\{x,y\},z\} +\{\{y,z\},x\} + \{\{z,x\},y\} \\
\{x,y\}\cdot z & = (y\cdot x)\cdot z - y\cdot (x\cdot z) -(x\cdot y)\cdot z+x\cdot (y\cdot z)
\label{post1} \\
x\cdot \{y,z\} & = \{x\cdot y,z\}+\{y,x\cdot z\} \label{post2}
\end{align}
for all $x,y,z \in V$.
\end{defi}

In particular a post-Lie algebra is a Lie algebra for the bracket $\{x,y\}$.
If this bracket is zero, then a post-Lie algebra is just a pre-Lie algebra.
Condition \eqref{post2} says, that the left multiplication $L(x)$, defined by
$L(x)y=x\cdot y$, is a derivation of the Lie algebra $(V,\{,\})$.

\begin{prop}
A post-Lie algebra $(V,\cdot,\{,\})$ has another associated Lie bracket, defined by
the formula
\begin{align}
[x,y] & = x\cdot y-y\cdot x +\{x,y\}. \label{post3}
\end{align}
\end{prop}

\begin{proof}
Of course we have $[x,y]=-[y,x]$. Furthermore,
using the above identities for all $x,y,z\in V$, we have
\begin{align*}
[x,[y,z]] & = [x,y\cdot z-z\cdot y+\{y,z\}] \\
 & = [x,y\cdot z]-[x,z\cdot y]+[x,\{y,z\}] \\
 & =x\cdot(y\cdot z)-(y\cdot z)\cdot x+\{x,y\cdot z \}\\
 & -x\cdot(z\cdot y)+(z\cdot y)\cdot x -\{x,z\cdot y \}\\
 & +x\cdot\{y,z\}-\{y,z\}\cdot x+\{x,\{y, z \}\}.
\end{align*}

It follows that
\begin{align*}
[x,[y,z]]+[y,[z,x]]+[z,[x,y]] & = \{x,\{y,z\}\} + \{y,\{z,x\}\}+\{z,\{x,y\}\} \\
 & + x\cdot \{y,z\}+\{z,x\cdot y \}-\{y,x\cdot z\} \\
 & + y\cdot \{z,x\}+\{x,y\cdot z \}-\{z,y\cdot x\} \\
 & + z\cdot \{x,y\}+\{y,z\cdot x \}-\{x,z\cdot y\} \\
 & + (y\cdot x)\cdot z -y\cdot (x\cdot z)-(x\cdot y)\cdot z+x\cdot(y\cdot z)-\{x,y\}\cdot z \\
 & + (x\cdot z)\cdot y -x\cdot (z\cdot y)-(z\cdot x)\cdot y+z\cdot(x\cdot y)-\{z,x\}\cdot y \\
 & + (z\cdot y)\cdot x -z\cdot (y\cdot x)-(y\cdot z)\cdot x+y\cdot(z\cdot x)-\{y,z\}\cdot x \\
 & = 0.
\end{align*}
This shows that the Jacobi identity is satisfied.
\end{proof}

\begin{prop}
The Lie bracket \eqref{post3} associated to a  post-Lie algebra $(V,\cdot,\{\, ,\})$
satisfies the identity
\begin{align}
[x,y]\cdot z & = x\cdot (y\cdot z)-y\cdot (x\cdot z), \label{post4}
\end{align}
i.e., the post-Lie algebra is a left-module over the Lie algebra $(V,[\, ,])$.
\end{prop}

\begin{proof}
For all $x,y,z\in V$ we have
\begin{align*}
[x,y]\cdot z & = (x\cdot y)\cdot z-(y\cdot x)\cdot z+\{x,y\}\cdot z \\
 & =  (x\cdot y)\cdot z-(y\cdot x)\cdot z + (y\cdot x)\cdot z - y\cdot (x\cdot z)
- (x\cdot y)\cdot z+x\cdot (y\cdot z) \\
 & = x\cdot (y\cdot z)-y\cdot (x\cdot z).
\end{align*}
\end{proof}

Let $(\Lg, [\, ,])$ and $(\Ln, \{\, ,\})$ be two Lie algebras with the same underlying vector
space $V$ over a field $k$. We call $(\Lg,\Ln)$ a {\it pair} of Lie algebras over $k$.
So as sets or vector spaces, $\Lg=\Ln=V$. In the sequel, when talking about a pair of Lie 
algebras $(\Lg,\Ln)$, we will always denote the Lie bracket in $\Lg$ with square brackets 
$[x,y]$ and those in $\Ln$ with curly brackets $\{x,y\}$ and the underlying vector space of 
both $\Lg$ and $\Ln$ will be denoted by $V$. 
Furthermore, when we write $\ad(x)$, we will always mean the adjoint operator in the Lie algebra $\Ln$. 

\begin{defi}\label{plt}
Let $(\Lg,\Ln)$ be a pair of Lie algebras. 
A {\it post-Lie algebra structure} on the pair $(\Lg,\Ln)$ is a $k$-bilinear product
$x\cdot y$ on $V$ satisfying the following identities:
\begin{align}
x\cdot y -y\cdot x & = [x,y]-\{x,y\} \label{post5}\\
[x,y]\cdot z & = x\cdot (y\cdot z) -y\cdot (x\cdot z) \label{post6}\\
x\cdot \{y,z\} & = \{x\cdot y,z\}+\{y,x\cdot z\} \label{post7}
\end{align}
for all $x,y,z \in V$.
\end{defi}

We obtain the following corollary by using \eqref{post6} and \eqref{post7}.

\begin{lem}\label{2.5}
The map $L\colon \Lg\ra \End(V)$ given by $x\mapsto L(x)$ is a linear representation of the
Lie algebra $\Lg$. Furthermore all operators $L(x)$ are Lie algebra derivations of $\Ln$.
\end{lem}

We can derive several consequences of the identities given in the definition.

\begin{lem}\label{lem1.5}
The axioms \eqref{post5},\eqref{post6},\eqref{post7} imply the following identities:
\begin{align}
\{x,y\}\cdot z & = (y\cdot x)\cdot z - y\cdot (x\cdot z) -(x\cdot y)\cdot z
+x\cdot (y\cdot z) \label{post8} \\[0.2cm]
z\cdot [x,y] & = z\cdot (x\cdot y)-z\cdot (y\cdot x)+z\cdot \{x,y\}
\label{post9} \\[0.2cm]
[x\cdot y,z]+[y,x\cdot z]-x\cdot [y,z] & = (x\cdot y)\cdot z-(x\cdot z)\cdot y
 + y\cdot (x\cdot z)-x\cdot(y\cdot z) \label{post10} \\
 & \, + \; x\cdot (z\cdot y)-z\cdot (x\cdot y) \nonumber \\[0.2cm]
x\cdot \{y,z\}+y\cdot\{z,x\}+z\cdot\{x,y\} & = \{[x,y],z\}+ \{[y,z],x\}
+\{[z,x],y\} \label{post11}\\[0.2cm]
\{x,y\}\cdot z+\{y,z\}\cdot x+\{z,x\}\cdot y & = \{[x,y],z\}+ \{[y,z],x\}+ \{[z,x],y\}
\label{post12}\\
 & \, +[\{x,y\},z]+[\{y,z\},x]+[\{z,x\},y]\nonumber
\end{align}
for all $x,y,z \in V$.
\end{lem}

\begin{proof}
Using \eqref{post5} and \eqref{post6} we obtain
\begin{align*}
\{x,y\}\cdot z & = ([x,y]-x\cdot y+y\cdot x)\cdot z \\
               & = [x,y]\cdot z-(x\cdot y)\cdot z+(y\cdot x)\cdot z \\
               & = (y\cdot x)\cdot z - y\cdot (x\cdot z) -(x\cdot y)\cdot z+x\cdot (y\cdot z).
\end{align*}
This gives \eqref{post8} which is just \eqref{post1}. Identity \eqref{post9} follows directly
by \eqref{post5}. Using \eqref{post5} and \eqref{post7} we obtain
\begin{align*}
x\cdot (y\cdot z) -x\cdot (z\cdot y) & = x\cdot ([y,z]-\{y,z\}) \\
               & = x\cdot [y,z]-x\cdot \{y,z\} \\
               & = x\cdot [y,z]-\{x\cdot y,z\}-\{y, x\cdot z\}\\
               & = x\cdot [y,z]-([x\cdot y,z]+z\cdot (x\cdot y)-(x\cdot y)\cdot z)\\
               & -([y,x\cdot z]+(x\cdot z)\cdot y-y\cdot (x\cdot z)).
\end{align*}

This gives \eqref{post10}. Using \eqref{post5},\eqref{post7} and the Jacobi identity for
$\{\, ,\}$ we have

\begin{align*}
0 & =  \{\{x,y\},z\}+  \{\{y,z\},x\} + \{\{z,x\},y\}\\
               & = \{ [x,y]-x\cdot y+y\cdot x,z\}+ \{ [y,z]-y\cdot z+z\cdot y,x\}+
\{ [z,x]-z\cdot x+x\cdot z,y\}\\
               & = \{[x,y], z\}-\{x\cdot y,z\}+\{y\cdot x,z\}+
                   \{[y,z], x\}-\{y\cdot z,x\}+\{z\cdot y,x\}\\
               & \,+ \{[z,x], y\}-\{z\cdot x,y\}+\{x\cdot z,y\}\\
               & = \{[x,y],z\}+ \{[y,z],x\}+\{[z,x],y\}-x\cdot \{y,z\}-y\cdot\{z,x\}-z\cdot\{x,y\}.
\end{align*}

This is \eqref{post11}. For \eqref{post12} use \eqref{post5} in the following way:

\begin{align*}
0 & =  \{\{x,y\},z\}+  \{\{y,z\},x\} + \{\{z,x\},y\}\\
  & = [\{x,y\},z] - \{x,y\}\cdot z+z\cdot \{x,y\}+[\{y,z\},x] - \{y,z\}\cdot x+x\cdot \{y,z\}\\
  & \, +[\{z,x\},y] - \{z,x\}\cdot y+y\cdot \{z,x\}.
\end{align*}
By applying \eqref{post11} the identity \eqref{post12} follows.
\end{proof}

\begin{cor}
Let $x\cdot y$ be a post-Lie algebra structure on $(\Lg,\Ln)$. Then $(V,\cdot,\{\, , \})$
is a post-Lie algebra with associated second Lie algebra $\Lg$.
Conversely, if $(V,\cdot,\{\, , \})$ is a post-Lie algebra, with associated second Lie algebra 
$\Lg=(V,[\, ,])$, then $x\cdot y$ is a post-Lie algebra structure on the pair $(\Lg,\Ln)$,
where $\Ln=(V,\{\, , \}).$
\end{cor}

\begin{proof}
The axioms of a post-Lie algebra structure $x\cdot y$ on $(\Lg,\Ln)$ imply the
conditions \eqref{post0}--\eqref{post2} and also imply that $\Lg$ is the associated Lie algebra.
Conversely, the defining axioms of a post-Lie algebra and the definition \eqref{post3} of the associated Lie 
algebra show $\Ln$ is really a Lie algebra and that $x\cdot y$ is a post-Lie algebra stucture on $(\Lg,\Ln)$.
\end{proof}

\begin{ex}\label{zero}
Suppose that the post-Lie algebra structure on $(\Lg,\Ln)$ is given by the zero product.
Then $(\Lg,[\, ,])=(\Ln,\{\, ,\})$.
\end{ex}
Indeed, $x\cdot y=0$ implies $[x,y]=\{x,y\}$ for all $x,y\in V$.

\begin{ex}\label{lsa}
If $\Ln$ is abelian, then a post-Lie algebra structure on $(\Lg,\Ln)$ corresponds to
a pre-Lie algebra structure on $\Lg$.
\end{ex}

If $\{x,y\}=0$ for all $x,y\in V$, then the conditions reduce to
\begin{align*}
x\cdot y-y\cdot x & = [x,y], \\
[x,y]\cdot z & = x\cdot (y\cdot z)-y\cdot (x\cdot z),
\end{align*}
i.e., $x\cdot y$ is a {\it pre-Lie algebra structure} on the Lie algebra $\Lg$. 
Conversely, a pre-Lie algebra structure on a Lie algebra $\Lg$ induces a post-Lie algebra structure 
on the pair of Lie algebras $(\Lg,\Ln)$, where $\Ln$ is the abelian Lie algebra on the same underlying vector 
space as $\Lg$.

\begin{ex}
If $\Lg$ is abelian,  then a post-Lie algebra structure on $(\Lg,\Ln)$ corresponds to
an LR-structure on $\Ln$.
\end{ex}

If $\Lg$ is abelian, then the conditions reduce to
\begin{align*}
x\cdot y-y\cdot x & = -\{x,y\} \\
x\cdot (y\cdot z)& = y\cdot (x\cdot z), \\
(x\cdot y)\cdot z & =(x\cdot z)\cdot y,
\end{align*}
i.e., $-x\cdot y$ is an {\it LR-structure} on the Lie algebra $\Ln$ (\cite{BU34}).
This follows from \eqref{post5}, \eqref{post6} and \eqref{post10}.
Conversely, an LR-structure on a Lie algebra $\Ln$ induces a post-Lie algebra structure 
on the pair of Lie algebras $(\Lg,\Ln)$, where $\Lg$ is the abelian Lie algebra on the same underlying vector 
space as $\Ln$.\\
If the Lie algebra $\Ln$ is complete, then we can say more on post-Lie algebra
structures on $(\Lg,\Ln)$. Recall that a Lie algebra $\Ln$ is called {\it complete}, if
$\Der(\Ln)=\ad(\Ln)$ and $Z(\Ln)=0$. 

\begin{lem}
Suppose that $x\cdot y$ is a post-Lie algebra structure on $(\Lg,\Ln)$, where
$\Ln$ is complete. Then there is a unique linear map
$\phi\colon V\ra V$ such that $x\cdot y=\{\phi(x),y\}$, i.e., satisfying $L(x)=\ad(\phi(x))$.
\end{lem}

\begin{proof}
For any $x\in V$, we have $L(x)\in \Der(\Ln)=\ad(\Ln)$ by Lemma $\ref{2.5}$. As $\Ln$ has trivial center, there is a 
unique element $\phi(x)\in \Ln$ such that $L(x)=\ad (\phi(x))$, which defines the map $\phi\colon V\ra V$. For $x,x',y\in V$ we have
\begin{align*}
\{\phi(x+x'),y\} & = (x+x')\cdot y \\
                 & = x\cdot y+x'\cdot y \\
                 &= \{\phi(x)+\phi(x'),y\}.
\end{align*}
It follows $\phi(x+x')=\phi(x)+\phi(x')$, because $\Ln$ has trivial center. In the same way we
obtain $\phi(\la x)=\la \phi(x)$, hence $\phi$ is linear.
\end{proof}

Inspired by the above, we now show the following result, which applies in particular for $\Ln$ being semisimple.

\begin{prop}
Let $(\Lg,\Ln)$ be a pair of Lie algebras such that $\Ln$ has trivial center. Let $\phi\in \End(V)$. Then
the product $x\cdot y=\{\phi(x),y \}$ is a post-Lie algebra structure on $(\Lg,\Ln)$
if and only if
\begin{align*}
\{\phi(x),y\}+\{x,\phi(y)\} & =[x,y]-\{x,y\},\\
\phi([x,y]) & = \{\phi(x),\phi(y)\} 
\end{align*}
for all $x,y\in V$.
\end{prop}

\begin{proof} Assume that $x\cdot y=\{\phi(x),y \}$ is a post-Lie algebra structure on $(\Lg,\Ln)$. Then 
the first identity follows immediately from \eqref{post5}. The second one follows from
\eqref{post6} and the Jacobi identity for $\Ln$. For $x,y,z\in V$ we have
\begin{align*}
\{ \phi([x,y]),z\} & = [x,y]\cdot z \\
                   & = x\cdot (y\cdot z)-y\cdot (x\cdot z) \\
                   & = x\cdot \{\phi(y),z\}-y\cdot \{\phi(x),z \} \\
                   & = \{\phi(x),\{\phi(y),z\}\}-\{\phi(y),\{\phi(x),z\}\} \\
                   & = \{\{\phi(x),\phi(y)\},z\}.
\end{align*}
Since $Z(\Ln)=0$ the claim follows, i.e., the map $\phi\colon \Lg\ra \Ln$ is a Lie algebra 
homomorphism.\\
Conversely, one can also show that when the two indentities are satisfied, the product 
 $x\cdot y=\{\phi(x),y \}$ does define a post-Lie algebra structure on $(\Lg,\Ln)$.
\end{proof}

The next result shows that we have a correspondence between post-Lie algebra
structures on $(\Lg,\Ln)$ and embeddings $\Lg \hookrightarrow \Ln\rtimes \Der(\Ln)$, where we recall that the 
Lie bracket on $ \Ln\rtimes \Der(\Ln)$ is given by
\[
[(x,D),(x',D')]=(\{x,x'\}+D(x')-D'(x),[D,D']).
\]

\begin{prop}\label{inj hom}
Let $x\cdot y$ be a post-Lie algebra structure on $(\Lg,\Ln)$. Then the map
\[
\phi\colon \Lg \ra \Ln\rtimes \Der(\Ln),\; x\mapsto (x,L(x))
\]
is an injective homomorphism of Lie algebras. Conversely any such embedding, with the identity
map on the first factor yields a post-Lie algebra structure onto $(\Lg,\Ln)$.
\end{prop}

\begin{proof}
Let $x\cdot y$ be a post-Lie algebra structure on $(\Lg,\Ln)$. We have
\begin{align*}
[\phi(x),\phi(y)] & = [(x,L(x)),(y,L(y))] \\
 & = (\{x,y\}+x\cdot y-y\cdot x,[L(x),L(y)]) \\
 & = ([x,y], L([x,y])) \\
 & = \phi([x,y]),
\end{align*}
where we have used \eqref{post5},\eqref{post6},\eqref{post7}.
Conversely, if we have a given embedding $\phi(x)=(x,L(x))$ with a derivation
$L(x)$, define $x\cdot y$ by $L(x)y$. Then the identities 
\eqref{post5},\eqref{post6},\eqref{post7} follow as above.
\end{proof}

We obtain the following result.

\begin{prop}\label{subalgebra}
There is a 1-1 correspondence between the post-Lie algebra structures on $(\Lg,\Ln)$ 
and the subalgebras $\Lh$ of $\Ln \rtimes \Der(\Ln)$ for which the projection 
$p_1\colon\Ln \rtimes \Der(\Ln) \rightarrow \Ln$ onto the first factor
induces a Lie algebra isomorphism of $\Lh$ onto $\Lg$.
\end{prop}
Note that as vector spaces $\Ln=V=\Lg$, so that $p_1$ can indeed be seen as a map onto $\Lg$.
\begin{proof}
Assume that there exists a post-Lie algebra structure on $(\Lg,\Ln)$, and denote by $\phi$ the
corresponding embedding as above. Then $\Lh=\im \phi=\{(x,L(x))\mid x\in \Lg\}$ 
is the Lie subalgebra corresponding to $\Lg$. It is obviously a subalgebra of $\Ln \rtimes \Der(\Ln)$ 
and $\phi$ induces an isomorphism of $\Lg$ onto $\Lh$. It is clear that the restriction of 
$p_1$ to $\Lh$ is the inverse of this isomorphism, and so is itself an isomorphism. 

Conversely, let $\Lh$ be a subalgebra of $\Ln \rtimes \Der(\Ln)$, for which ${p_1}_{\mid \Lh}:\Lh \rightarrow \Lg$ 
is an isomorphism. Then the inverse map 
\[ \phi = ({p_1}_{\mid \Lh})^{-1}:\Lg \rightarrow \Lh\le \Ln \rtimes \Der(\Ln)\]
is an embedding inducing the identity on the first factor. Hence, by proposition~\ref{inj hom}, $\phi$ determines 
a post-Lie algebra structure on $(\Lg, \Ln)$. 

In the above, we showed how to assign a subalgebra $\Lh$ to a post-Lie algebra structure and vice versa. It is obvious that 
these two operations are each others inverse.
\end{proof}

\begin{rem}\label{rem subalgebra}
Given a Lie algebra $\Ln$, let $\Lh$ be any subalgebra of $\Ln\rtimes \Der(\Ln)$ for which the projection $p_1$ onto the first factor 
is a bijection. Then, for any $x\in \Ln$, there is exactly one $L(x)\in \Der(\Ln)$ such that $(x,L(x))\in \Lh$.
We can define a new Lie bracket on $\Ln$ by 
$$[x,y]:= p_1([(x,L(x)),(y,L(y))])$$
and denote the corresponding Lie algebra by $\Lg$.
Now, $\phi:\Lg\rightarrow \Lh, \; x \mapsto (x,L(x))$ is an isomorphism of Lie algebras, 
and $x\cdot y := L(x)y$ is the post-Lie algebra structure on $(\Lg,\Ln)$ corresponding to $\Lh$.
\end{rem}

In the special case where $\Ln$ is semisimple we can say more on the above 1-1 correspondence.
Then $\Der(\Ln)=\ad (\Ln)=\Ln$, and the Lie algebra $\Ln  \rtimes \Der(\Ln)=\Ln \rtimes \Ln$
is isomorphic to the direct sum $\Ln\oplus \Ln$. Indeed, the map $\psi\colon \Ln \rtimes \Ln
\ra \Ln\oplus \Ln$, $(x,y)\mapsto (x+y,y)$ is a Lie algebra isomorphism.

\begin{prop}\label{semisimple2}
Let $\Ln$ be a semisimple Lie algebra. Then there is a 1-1 correspondence 
between the post-Lie algebra structures on $(\Lg,\Ln)$ and the 
subalgebras $\Lh$ of $\Ln \oplus \Ln$ for which the map $p_1-p_2 \colon \Ln\oplus\Ln\rightarrow
\Ln : (x,y)\mapsto x-y$ induces an isomorphism of $\Lh$ onto $\Lg$. 
Here $p_i:\Ln\oplus \Ln\rightarrow \Ln$ denotes projection onto the $i$-th factor ($i=1,2$).
\end{prop}

\begin{proof}
This follows from proposition $\ref{subalgebra}$ by noting that a subalgebra $\Lh$ of
$\Ln\rtimes \Ln$ for which $p_1$ induces an isomorphism of $\Lh$ on $\Lg$ corresponds via $\psi$ to
a subalgebra $\Lh'=\psi(\Lh)$ of $\Ln\oplus \Ln$ such that $p_1-p_2\colon \Ln\oplus \Ln \ra \Ln$,
$(x,y)\mapsto x-y$ induces an isomorphism of $\Lh'$ on $\Lg$. This is visualized by the
following diagram:
\[ 
\xymatrix{
\Lh\leq \Ln\rtimes\Ln \ar[r]^{ \ \ \ p_1}\ar@<1ex>[d]^{\psi}  &\ \ \  \Ln 
\\ \Lh'\leq\Ln\oplus\Ln \ar@<1ex>[u]^{\psi^{-1}} \ar@{-->}[ur]_{p_1-p_2}. & \\ } 
\]
\end{proof}

We conclude this section by showing how post-Lie algebra structures arise naturally
in the study of NIL-affine actions of Lie groups. We say here that a post-Lie algebra structure 
is {\it complete} if all left multiplications are nilpotent.

\begin{thm}
Let $G$ and $N$ be connected, simply connected nilpotent Lie groups with associated Lie algebras
$\Lg$ and $\Ln$. Then there exists a simply transitive NIL-affine action of $G$ on $N$ if and only if
there is a Lie algebra $\Lg'\simeq \Lg$, with the same underlying vector space as $\Ln$,
such that the pair of Lie algebras $(\Lg',\Ln)$ 
admits a complete post-Lie algebra structure.
\end{thm}

\begin{proof}
Let $G$ and $N$ be connected, simply connected nilpotent Lie groups with corresponding Lie algebras 
respectively $\Lg$ and $\Ln$. Let $\rho:G\rightarrow \Aff(N)$ be a representation with corresponding 
differential $d\rho:\Lg\rightarrow \Ln\rtimes\Der (\Ln): x\mapsto (t(x),D(x))$. Recall that 
$\La\Lf\Lf(\Ln)=\Ln\rtimes\Der (\Ln)$ is the Lie algebra of the Lie group $\Aff(N)$. Then $\rho$ 
induces a simply transitive NIL-affine action of $G$ on $N$ if and only if $d\rho$ is a complete 
NIL-affine structure on $\Lg$ (\cite[Theorem 3.1]{BDD}). This means that $d\rho$ is a Lie algebra homomorphism such that 
$t:\Lg\rightarrow \Ln:x\mapsto t(x)$ is bijective and such that $D(x)$ is nilpotent for all $x\in \Lg$.

Now suppose we have a complete post-Lie algebra structure on a pair of Lie algebras $(\Lg',\Ln)$ where $\Lg'$ is isomorphic
to $\Lg$, say via $\psi:\Lg\rightarrow \Lg'$. Hence, by 
proposition \ref{inj hom}, we have that
\[
\phi\colon \Lg' \ra \Ln\rtimes \Der(\Ln),\; x\mapsto (x,L(x))
\]
is an injective Lie algebra homomorphism such that all $L(x)$ are nilpotent. The composition 
$\phi \circ \psi:\Lg \rightarrow \Ln\rtimes \Der(\Ln)$ is then clearly a complete NIL-affine structure on $\Lg$.

For the converse statement suppose that $d\rho:\Lg\rightarrow \Ln\rtimes\Der (\Ln): 
x\mapsto (t(x),D(x))$ is a complete NIL-affine structure on $\Lg$. 
Then $\Lh=d\rho(\Lg)$ is a Lie subalgebra of $\Ln\rtimes\Der (\Ln)$ 
for which the projection on the first factor induces a bijection of $\Lh$ on $\Ln$ since 
$t$ is bijective. By remark \ref{rem subalgebra} this gives rise to a post-Lie algebra structure on 
$(\Lg',\Ln)$, where the Lie bracket on $\Lg'$ is given by 
\[
(x,y)= p_1([(x,D(t^{-1}(x))),(y,D(t^{-1}(y)))]),
\] 
and $p_1$ is the projection on the first factor. The left multiplications are given by 
$L(x)=D(t^{-1}(x))$, so these are all nilpotent and hence the post-Lie structure is complete. 
Note that $\Lg$ is isomorphic to $\Lh$ which is in its turn isomorphic to $\Lg'$. 
\end{proof}

\section{Special cases and examples}

Before studying more structural results concerning post-Lie algebras in the next section, it might be
useful to present some obvious examples. 

\begin{prop} Suppose $(\Lg,\Ln)$ is a pair of Lie algebras and 
let $\la\not\in \{0,1 \}$. Then $x\cdot y=\la [x,y]$ defines a post-Lie algebra structure 
on $(\Lg,\Ln)$ if and only if $\{x,y\}=(1-2\la)[x,y]$, and both
$\Lg$ and $\Ln$ are nilpotent of class at most 2. 
\end{prop}

\begin{proof}
Suppose that $x\cdot y=\la [x,y]$ defines a  post-Lie algebra structure on $(\Lg,\Ln)$. Then
\eqref{post5} implies $\{x,y\}=(1-2\la)[x,y]$. By \eqref{post6} and the Jacobi identity for $\Lg$ 
we obtain
\begin{align*}
\la [[x,y],z] & = [x,y]\cdot z \\
              & = x\cdot (y\cdot z)-y\cdot (x\cdot z) \\
              & = \la^2[x,[y,z]]-\la^2[y,[x,z]] \\
              & = \la^2[[x,y],z].
\end{align*}
Because $\la\neq 0,1$ this yields $[[x,y],z]=\{\{x,y\},z\}=(x\cdot y)\cdot z=0$. \\
Conversely, let $(\Lg,\Ln)$ be a pair of nilpotent Lie algebras of class $\leq 2$ with
$\{x,y\}=(1-2\la)[x,y]$, and $x\cdot y=\la [x,y]$.  Obviously, the identities \eqref{post5}, \eqref{post6} are satisfied.
To show \eqref{post7}, we use that $\{x,y\}=\mu \, x\cdot y$ with $\mu=\frac{1-2\la}{\la}$:
\begin{align*}
x\cdot \{y,z\} & = \mu\, x\cdot (y\cdot z)\\
               & = \mu\la^2 [x,[y,z]] \\
               & = \mu\la^2 [[x,y],z]+\mu\la^2[y,[x,z]]\\
               & = \mu\, (x\cdot y)\cdot z+ \mu\, y\cdot (x\cdot z)\\
               & = \{x\cdot y,z\}+\{y,x\cdot z\}.
\end{align*}
\end{proof}

\medskip

Note that for $\la=\frac{1}{2}$ we have  
$x\cdot y=\frac{1}{2}[x,y]$ and $\{x,y\}=0$. Hence $\Ln$ is abelian, and the product defines
a pre-Lie algebra structure (even a Novikov structure) on $\Lg$ (still assuming $\Lg$ is nilpotent 
of class $\leq2$), see example $\ref{lsa}$. 

\medskip

It is easy to discuss the cases $\la=0,1$ which we have excluded above.
For $\la=0$ we have the zero product $x\cdot y=0$ with $[x,y]=\{x,y\}$, see example
$\ref{zero}$. It gives a trivial post-Lie algebra structure on $(\Lg,\Lg)$ for any $\Lg$. 
For $\la=1$ we have $x\cdot y=[x,y]=-\{x,y\}$. This defines a post-Lie algebra
structure on $(\Lg,-\Lg)$ for any $\Lg$. 

\begin{rem}
We have an analogous result for post-Lie algebra structures defined by $x\cdot y=\mu \{x,y\}$.
For $\mu\not\in \{0,-1 \}$ this means that $[x,y]=(1+2\mu)\{x,y\}$, and both
$\Lg$ and $\Ln$ are nilpotent of class at most 2. In fact, we obtain the same post-Lie algebra structures as
above, except for the case $\mu=-\frac{1}{2}$, where $\Lg$ is abelian, and $x\cdot y=-\frac{1}{2}\{x,y\}$
defines an LR-structure on a nilpotent Lie algebra $\Ln$ of class $\leq 2$.
\end{rem}

Another special case arises if $\{x,y\}=\rho [x,y]$ for some nonzero scalar $\rho$.

\begin{ex}
Let $\rho\not\in \{0,1 \}$ and $\{x,y\}=\rho [x,y]$. Then $x\cdot y$ defines a post-Lie 
algebra structure on $(\Lg,\Ln)$ if and only if 
\begin{align*}
x\cdot y-y\cdot x & = (1-\rho)[x,y] \\
(1-\rho)(x\cdot (y\cdot z)-y\cdot (x\cdot z)) & = (x\cdot y)\cdot z- (y\cdot x)\cdot z \\
x\cdot (y\cdot z)-y\cdot (x\cdot z) & = (x\cdot y)\cdot z-z\cdot (x\cdot y)-
(x\cdot z)\cdot y+x\cdot(z\cdot y).
\end{align*} 
\end{ex}

This says that $x\cdot y$ is a certain deformed pre-Lie product on $\Lg$. In general, 
it seems difficult to classify such products. For semisimple Lie algebras however it is
possible, see \cite{BD2}. \\
There is also the interesting case $\rho=1$, i.e., $\{x,y\}=[x,y]$. 

\begin{ex}
Let $\{x,y\}=[x,y]$. Then $x\cdot y$ defines a post-Lie algebra structure on $(\Lg,\Ln)$ 
if and only if 
\begin{align*}
x\cdot y & = y\cdot x \\
[x,y]\cdot z & = x\cdot (y \cdot z)- y \cdot (x\cdot z) \\
x\cdot [y,z] & = [x\cdot y, z]+[y,x\cdot z].
\end{align*} 
\end{ex}

Hence $x\cdot y$ is a commutative product on $\Lg$ such that the operators
$L(x)$ are derivations, and $L([x,y])=[L(x),L(y)]$. For semisimple Lie algebras this
can be classified, see \cite{BD2}. In general, this seems to be difficult.
Already for the Heisenberg Lie algebra $\Ln_3(\C)$ there are many such structures:
let $(e_1,e_2,e_3)$ be a basis of $\C^3$ and define the non-zero Lie brackets of $\Lg$ and
$\Ln$ by $[e_1,e_2]=e_3$,  $\{e_1,e_2\}=e_3$.

\begin{ex}
Let $\Lg=\Ln=\Ln_3(\C)$ and $\al,\be,\ga\in \C$ with $\be \neq 0$. Then
\begin{align*}
e_1\cdot e_1 & = e_1-\be^{-1}e_2+\al e_3 \\
e_1\cdot e_2 = e_2 \cdot e_1 & = \be e_1-e_2+\frac{\ga+\al\be^2}{2\be}e_3\\
e_2\cdot e_2 & = \be^2e_1-\be e_2+\ga e_3
\end{align*}
defines a commutative post-Lie algebra structure on $(\Lg,\Ln)$, where we did not write down the 
zero products between basis vectors.
\end{ex}

In the following we want to classify all complex two-dimensional post-Lie algebras 
$(V,\cdot,\{\, ,\})$. Of course, two post-Lie algebras $(V,\cdot,\{\, ,\})$ and $(W,\cdot,\{\, ,\})$ 
are isomorphic if and only if there
exists a bijective linear map $\phi:V\rightarrow W$, which preserves both products:
\begin{align*}
\phi(x \cdot y) & = \phi(x) \cdot \phi(y), \\ 
\phi(\{x,y\})   & = \{\phi(x)  ,\phi(y) \}, 
\end{align*}
for all $x,y\in V$. It is obvious that isomorphic post-Lie algebras have 
isomorphic associated Lie algebras $\Lg$ and $\Ln$.

Now, if $(V,\cdot,\{\, ,\})$ is a two-dimensional complex post-Lie algebra, then the 
associated Lie algebras are either $\C^2$, or $\Lr_2(\C)$, the non-abelian Lie algebra of 
dimension $2$. In our classification, we distinguish between four cases, depending on the isomorphism 
types of these associated Lie algebras.  \\[0.2cm]
{\it Case 1}: $(\Lg,[\, ,])$ and $(\Ln,\{\, ,\})$ are abelian. \\[0.2cm]
There is a basis $(e_1,e_2)$ of $V$ such that $[e_1,e_2]=\{e_1,e_2\}=0$. Then \eqref{post5}
says that $x\cdot y=y\cdot x$ for all $x,y\in V$, \eqref{post6} says that
$x\cdot (y\cdot z)=y\cdot (x\cdot z)$ for all $x,y,z\in V$, and \eqref{post7} says $0=0$.
This implies
\[
x\cdot (z\cdot y)=x\cdot (y\cdot z)=y\cdot (x\cdot z)=(x\cdot z)\cdot y,
\]
so that post-Lie algebra structures on $(\C^2,\C^2)$ correspond to $2$-dimensional
commutative and associative algebras. The classification is well known, see for example
\cite{BU36}:
\vspace*{0.5cm}
\begin{center}
\begin{tabular}{c|c|c|c}
$V$ & Products & $[\, ,]$ & $\{\, ,\}$ \\
\hline
$V_1$ & $-$ & $[e_1,e_2]=0$ & $\{e_1,e_2\}=0$\\
\hline
$V_2$ & $e_1\cdot e_1=e_1$ & $[e_1,e_2]=0$ & $\{e_1,e_2\}=0$\\
\hline
$V_3$ & $e_1\cdot e_1=e_1,\; e_2\cdot e_2=e_2$ & $[e_1,e_2]=0$ & $\{e_1,e_2\}=0$\\
\hline
$V_4$ & $e_1\cdot e_2=e_1,\;e_2\cdot e_1=e_1$,  & $[e_1,e_2]=0$ & $\{e_1,e_2\}=0$\\
      & $e_2\cdot e_2=e_2$  & &\\
\hline
$V_5$ & $e_2\cdot e_2=e_1$  & $[e_1,e_2]=0$ & $\{e_1,e_2\}=0$ \\
\end{tabular}
\end{center}
\vspace*{0.5cm}
{\it Case 2}: $(\Lg,[\, ,])$ is abelian, and $(\Ln,\{\, ,\})$ is not abelian. \\[0.2cm]
We may choose a basis  $(e_1,e_2)$ of $V$ such that $[e_1,e_2]=0$ and
$\{e_1,e_2\}=-e_1$. Then post-Lie algebra structures on $(\C^2,\Lr_2(\C))$ are just
LR-structures on $\Ln$, which we have classified in \cite{BDD}:
\vspace*{0.5cm}
\begin{center}
\begin{tabular}{c|c|c|c}
$V$ & Products &  $[\, ,]$ & $\{\, ,\}$ \\
\hline
$V_6$ & $e_1\cdot e_1=e_1,\; e_2\cdot e_1=-e_1$ & $[e_1,e_2]=0$ & $\{e_1,e_2\}=-e_1$\\
\hline
$V_7$ & $e_1\cdot e_2=e_1$ & $[e_1,e_2]=0$ & $\{e_1,e_2\}=-e_1$\\
\hline
$V_8$ & $e_2\cdot e_1=-e_1$ & $[e_1,e_2]=0$ & $\{e_1,e_2\}=-e_1$\\
\end{tabular}
\end{center}
Note that $(V_8,\cdot)$ is also an LSA (left-symmetric algebra). \\[0.3cm]
{\it Case 3}: $(\Lg,[\, ,])$ is not abelian, and $(\Ln,\{\, ,\})$ is abelian. \\[0.2cm]
We may choose a basis  $(e_1,e_2)$ of $V$ such that $[e_1,e_2]=e_1$ and
$\{e_1,e_2\}=0$. Then post-Lie algebra structures on $(\Lr_2(\C),\C^2)$ are just
LSA-structures (pre-Lie algebra structures) on $\Lg$, which we have classified in \cite{BU36}:
\vspace*{0.5cm}
\begin{center}
\begin{tabular}{c|c|c|c}
$V$ & Products &  $[\, ,]$ & $\{\, ,\}$ \\
\hline
$V_9(\al)$ & $e_2\cdot e_1=-e_1,\; e_2\cdot e_2=\al e_2$  & $[e_1,e_2]=e_1$ & $\{e_1,e_2\}=0$\\
\hline
$V_{10}(\be)$ & $e_1\cdot e_2=\be e_1,\; e_2\cdot e_1=(\be-1)e_1$, & $[e_1,e_2]=e_1$ & $\{e_1,e_2\}=0$\\
$\be \neq 0$ & $e_2\cdot e_2=\be e_2$ &  \\
\hline
$V_{11}$  & $e_2\cdot e_1=-e_1,\; e_2\cdot e_2=e_1-e_2$ & $[e_1,e_2]=e_1$ & $\{e_1,e_2\}=0$\\
\hline
$V_{12}$  & $e_1\cdot e_1=e_2,\; e_2\cdot e_1=-e_1$ & $[e_1,e_2]=e_1$ & $\{e_1,e_2\}=0$\\
       & $e_2\cdot e_2=-2e_2$ & \\
\hline
$V_{13}$  & $e_1\cdot e_2=e_1,\; e_2\cdot e_2=e_1+e_2$ & $[e_1,e_2]=e_1$ & $\{e_1,e_2\}=0$\\
\end{tabular}
\end{center}
\vspace*{0.5cm}
Note that $(V_9(0),\cdot)$ is also a complete LR-algebra.  \\[0.3cm]
{\it Case 4}: $(\Lg,[\, ,])$ and $(\Ln,\{\, ,\})$ are not abelian. \\[0.2cm]
We may choose a basis $(e_1,e_2)$ of $V$ such that $[e_1,e_2]=\al_1e_1+\al_2e_2$ 
with $(\al_1,\al_2)\neq (0,0)$, and $\{e_1,e_2\}=e_1$. Note that we cannot make further 
assumptions on $\al_1$ or $\al_2$. On the other hand, the
conditions \eqref{post5}, \eqref{post6}, \eqref{post7} become very restrictive.
They immediately imply that $\al_2=0$, and hence $\al_1\neq 0$. 
We can easily list {\it all} possible products in this case, regardless of post-Lie algebra
isomorphism. We obtain two families of algebras, the first one given by
\[
e_2\cdot e_1  = (1-\al_1)e_1,\; e_2\cdot e_2=\al e_1, \\
\]
where $\al$ is an arbitrary complex number, and the second one given by
\[
e_1\cdot e_2 = -e_1,\; e_2\cdot e_1=-\al_1e_1,\; e_2\cdot e_2=\be e_1, \\
\]
where $\be$ is an arbitrary complex number. Let $\phi=(\phi_{ij})\in \End(V)$. It 
is an automorphism of $\Ln$ if and only if $\phi_{21}=0$, $\phi_{22}=1$ and $\det(\phi)
=\phi_{11}\neq 0$. Applying these automorphisms we obtain the classification of the above
products as post-Lie algebras:
\vspace*{0.5cm}
\begin{center}
\begin{tabular}{c|c|c|c}
$V$ & Products  &  $[\, ,]$ & $\{\, ,\}$ \\
\hline
$V_{14,\al_1},\, \al_1\neq 0$ & $e_2\cdot e_1=(1-\al_1)e_1$  & $[e_1,e_2]=\al_1e_1$ & $\{e_1,e_2\}=e_1$\\
\hline
$V_{15}$ & $e_2\cdot e_2=e_1$  & $[e_1,e_2]=e_1$ & $\{e_1,e_2\}=e_1$\\
\hline
$V_{16,\al_1},\, \al_1\neq 0$ & $e_1\cdot e_2=-e_1,\; e_2\cdot e_1=-\al_1e_1$ 
& $[e_1,e_2]=\al_1e_1$ & $\{e_1,e_2\}=e_1$ \\
\hline
$V_{17}$  & $e_1\cdot e_2=-e_1,\; e_2\cdot e_1=e_1$ & $[e_1,e_2]=-e_1$ & $\{e_1,e_2\}=e_1$\\
         & $e_2\cdot e_2=e_1$ & &\\
\end{tabular}
\end{center}
\vspace*{0.5cm}
The algebras $(V_{14,\al_1},\cdot)$ and $(V_{15},\cdot)$ are LR and LSA, but the algebras
$(V_{16,\al_1},\cdot)$ and $(V_{17},\cdot)$ are not. They satisfy a new identity, namely
\begin{align*}
x\cdot (y\cdot z) + y\cdot (x\cdot z) + z\cdot (x\cdot y) & =
(y\cdot z)\cdot x + (x\cdot z)\cdot y + (x\cdot y)\cdot z
\end{align*}
for all $x,y,z\in V$.

\section{Structure results for $\Lg$ and  $\Ln$}

The existence of post-Lie algebra structures on a pair of Lie algebras
$(\Lg,\Ln)$ imposes certain algebraic conditions on $\Lg$ and $\Ln$.
In particular, the algebraic structures of $\Lg$ and $\Ln$ depend on each other
in a certain way. We will show, for example, that if $\Lg$ is nilpotent and $(\Lg,\Ln)$ admits
a post-Lie algebra structure, then $\Ln$ must be solvable. But first we study the situation
in which $\Ln$ is 2--step nilpotent.

\begin{lem}\label{2step nilp}
Let $\Ln$ be a $2$-step nilpotent Lie algebra and $\Lm$ be the abelian
Lie algebra with the same underlying vector space as $\Ln$. Then
\begin{align*}
\psi \colon \Ln \rtimes \Der(\Ln) \ra  \Lm \rtimes \Der(\Lm)=\Lm \rtimes \Lg\Ll(\Lm),\\
(x,D)\mapsto \bigl(x,\frac{1}{2}\ad(x)+D\bigr)
\end{align*}
is an embedding of Lie algebras.
\end{lem}

\begin{proof}
The map is obviously an injective linear map. It remains to show that it is a Lie algebra 
homomorphism. Since  $\Ln$ is $2$-step nilpotent we have $\ad ([x,y])=0$.
Using the identity $[D,\ad(x)]=\ad(D(x))$ we obtain

\begin{eqnarray*}
\psi([(x_1,D_1),(x_2,D_2)])&=& \psi(([x_1,x_2]+D_1(x_2)-D_2(x_1),[D_1,D_2]))\\[0.1cm]
&=& \Big([x_1,x_2]+D_1(x_2)-D_2(x_1),\frac{1}{2}\ad(D_1(x_2))\\[0.1cm]
& & -\frac{1}{2}\ad(D_2(x_1))+[D_1,D_2]\Big).
\end{eqnarray*}

On the other hand we have

\begin{eqnarray*}
[\psi((x_1,D_1)),\psi((x_2,D_2))]&=& \Big[\big(x_1,\frac{1}{2}\ad (x_1)+D_1\big),
\big(x_2,\frac{1}{2}\ad (x_2)+D_2\big)\Big]\\[0.1cm]
&=& \Big(\frac{1}{2}\ad (x_1)(x_2)-\frac{1}{2}\ad (x_2)(x_1)+D_1(x_2)-D_2(x_1), \\
& & \Big[\frac{1}{2}\ad (x_1)+D_1,\frac{1}{2}\ad (x_2)+D_2\Big]\Big)\\[0.1cm]
&=& \Big([x_1,x_2]+D_1(x_2)-D_2(x_1),\frac{1}{2}[\ad (x_1),D_2]\\
& & +\frac{1}{2}[D_1,\ad (x_2)]+[D_1,D_2]\Big)\\[0.1cm]
&=& \Big([x_1,x_2]+D_1(x_2)-D_2(x_1),\frac{1}{2}\ad (D_1(x_2))\\
& & -\frac{1}{2}\ad (D_2(x_1))+[D_1,D_2]\Big).
\end{eqnarray*}
\end{proof}

\begin{prop}\label{prop4.2}
Suppose that there exists a post-Lie algebra structure on $(\Lg,\Ln)$, where $\Ln$ is 
$2$-step nilpotent. Then $\Lg$ admits a pre-Lie algebra structure. In particular, $\Lg$ is
not semisimple.
\end{prop}

\begin{proof}
Proposition \ref{inj hom} and the above lemma imply that also $(\Lg,\Lm)$ admits a 
post-Lie algebra structure. Since $\Lm$ is abelian, $\Lg$ admits a pre-Lie algebra structure, 
see example $\ref{lsa}$. A complex semisimple Lie algebra does not admit a pre-Lie algebra
structure, see \cite{BU24}, \cite{BD2}.
\end{proof}

If we assume that $\Lg$ is nilpotent, we obtain the following result.

\begin{prop}\label{g nilpotent}
Suppose that there exists a post-Lie algebra structure on $(\Lg,\Ln)$, where $\Lg$ is 
nilpotent. Then $\Ln$ is solvable.
\end{prop}

\begin{proof}
Consider the map $$ \phi\colon \Lg \ra \Ln\rtimes \Der(\Ln),\; x\mapsto (x,L(x))$$ 
induced by the post-Lie algebra structure. Then $\Lh=L(\Lg)$ is a nilpotent Lie algebra. 
We claim that $\Ln\rtimes \Lh=\phi(\Lg) \oplus \Lh$. Indeed, for  $(x,y)\in \Ln\rtimes \Lh$ we have
$$(x,y)-\phi(x)=(x,y)-(x,L(x))=(0,y-L(x)).$$
Hence $(x,y)=\phi(x)+(0,y-L(x))\in \phi(\Lg)\oplus \Lh$. Conversely, for $(x,y)\in \phi(\Lg)
\oplus \Lh$ there exist $a,b\in \Lg$ such that
$$(x,y)=\phi(a)+(0,L(b))=(a,L(a)+L(b))=(a,L(a+b))\in \Ln\rtimes \Lh.$$
It follows that $\Ln\rtimes \Lh$ is the direct sum of two 
nilpotent Lie algebras. Goto \cite{GOT} has shown that the sum of two nilpotent Lie algebras
is solvable. Hence $\Ln\rtimes \Lh$ is solvable, and so $\Ln$ itself is solvable.
\end{proof}

\begin{prop}\label{solv non nilp}
Suppose that there exists a post-Lie algebra structure on $(\Lg,\Ln)$, where $\Ln$ is
solvable and non-nilpotent. Then $\Lg$ is not perfect. 
\end{prop}

\begin{proof}
By assumption the nilradical  $\nil(\Ln)$ of $\Ln$ is different from $\Ln$. We have to
show that $\Lg\neq  [\Lg,\Lg]$.
For all $x\in \Ln$ the left multiplication $L(x)$ is a derivation of $\Ln$. 
Any derivation $D$ satisfies $D(\rad (\Ln))\subseteq \nil(\Ln)$. 
Since $\Ln$ is solvable we have $D(\Ln)\subseteq \nil(\Ln)$. In particular we have 
$\Ln\cdot \Ln\subseteq\nil(\Ln)$. It follows that for any $x,y\in \Lg$ we have
$$[x,y]=x\cdot y-y\cdot x+\{x,y\}\in \nil(\Ln).$$
This implies $[\Lg,\Lg]\subseteq \nil(\Ln) \varsubsetneq \Ln=\Lg$ as vector spaces, 
so that $\Lg\neq  [\Lg,\Lg]$.
\end{proof}

\begin{cor}\label{g-sl2}
Suppose that there exists a post-Lie algebra structure on $(\Lg,\Ln)$ with
$\Lg=\Ls\Ll_2(\C)$. Then $\Ln$ is isomorphic to $\Ls\Ll_2(\C)$.
\end{cor}

\begin{proof}
Suppose that $\Ln$ is nilpotent. Because of $\dim (\Ln)=3$, this would mean that 
it is $2$-step nilpotent or abelian. This is a contradiction to proposition $\ref{prop4.2}$ and
example $\ref{lsa}$. On the other hand, $\Ln$ cannot be solvable, non-nilpotent by 
proposition $\ref{solv non nilp}$. It follows that $\Ln$ is isomorphic to  $\Ls\Ll_2(\C)$.
\end{proof}

On a pair of simple Lie algebras $(\Lg,\Ln)$ only 
trivial post-Lie algebra structures are possible:

\begin{prop}\label{gn-simple}
Let $x\cdot y$ be a post-Lie algebra structure on $(\Lg,\Ln)$, where both $\Lg$ and
$\Ln$ are simple. Then either $x\cdot y=0$ for all $x$ and $y$ and $[x,y]=\{x,y\}$,
or $x\cdot y=[x,y]=-\{x,y\}$. 
\end{prop}

\begin{proof}
By proposition $\ref{semisimple2}$, any such post-Lie algebra structure corresponds
to a subalgebra $\Lh$ of  $\Ln\oplus\Ln$ for which the map $p_1-p_2\colon \Ln\oplus \Ln\ra \Ln$
induces an isomorphism of $\Lh$ onto $\Lg$. Since $\Lg$ is simple, $\Lh$ is simple too.
Both projections maps $p_1$ and $p_2$ are Lie algebra homomorphisms onto $\Ln$. Hence the kernels
$\ker(p_1(\Lh))$ and $\ker(p_2(\Lh))$ are ideals in $\Lh$, and hence must be either $0$
or $\Lh$. This yields three possible cases: \\[0.2cm]
{\it Case 1: $p_2(\Lh)=0$}: Then we have $\Lh=\{(x,0)\mid x\in \Ln\}$. Because
$L(x)=\ad (0)=0 $ for all $x\in \Ln$ we have $x\cdot y=0$ and $[x,y]=\{x,y\}$ for all $x,y\in \Ln$.
In particular, $\Lg=\Ln$.\\[0.2cm]
{\it Case 2: $p_1(\Lh)=0$}: Then we have $\Lh=\{(0,x)\mid x\in \Ln\}$, and $L(x)=-\ad (x)$.
It follows that $[x,y]=-\{x,y\}$ for all $x,y\in \Ln$. Hence $\Lg=-\Ln$.\\[0.2cm]
{\it Case 3: $p_1(\Lh)\neq 0$ and $p_2(\Lh)\neq 0$}: Then we have $\ker({p_1}_{|\Lh})=\ker({p_2}_{|\Lh})=0$. 
Hence $p_1$ and $p_2$ are both bijective when restricted to $\Lh$.
This implies that there is a bijective linear map $\phi:\Ln\rightarrow\Ln$ such that 
$\Lh=\{(x,\phi(x))\mid x\in \Ln\}$. As $\Lh$ is a subalgebra of $\Ln\oplus\Ln$, we know that 
$[(x,\phi(x)),(y,\phi(y))]\in \Lh$ for all $x,y\in \Ln$. 
So $(\{x,y\},\{\phi(x),\phi(y)\})=(z,\phi(z))$ for some $z\in\Ln$. It follows that $z=\{x,y\}$ 
and hence 
$$\phi(\{x,y\})=\phi(z)=\{\phi(x),\phi(y)\}.$$
This shows that $\phi\in\Aut(\Ln)$. By a result of Jacobson \cite{JAC}, $\la=1$ must be an 
eigenvalue of $\phi$. But then $p_1-p_2: \Lh\rightarrow\Lg:(x,\phi(x))\mapsto x-\phi(x)$ 
cannot be an isomorphism. This is a contradiction. Hence this case cannot occur.
\end{proof}

The next result classifies the possible $3$-dimensional Lie algebras $\Lg$ for which
the pair $(\Lg,\Ln)$ with $\Ln=\Ls\Ll_2(\C)$ admits a post-Lie algebra structure.
We denote by $\Lr_{3,\la}(\C)$ the series of solvable, non-nilpotent Lie algebras
with basis $(e_1,e_2,e_3)$ and Lie brackets $[e_1,e_2]=e_2,\; [e_1,e_3]=\la e_3$.
Here $\la\in \C$ is a parameter. For $\la=0$ we obtain the decomposible Lie algebra
$\Lr_2(\C)\oplus \C$.

\begin{prop}\label{n-sl2}
Suppose that there exists a post-Lie algebra structure on $(\Lg,\Ln)$, where $\Ln$ is 
$\Ls\Ll_2(\C)$. Then $\Lg$ is isomorphic to $\Ls\Ll_2(\C)$ or to one of the 
Lie algebras $\Lr_{3,\la}(\C)$ for $\la\neq -1$. Moreover, all these possibilities do
occur.
\end{prop}

\begin{proof}
Assume first that there exists a post-Lie algebra structure on $(\Lg,\Ln)$ for
$\Ln=\Ls\Ll_2(\C)$. Then $\Lg$ cannot be nilpotent by proposition $\ref{g nilpotent}$.
As we have seen, the case $\Lg=\Ls\Ll_2(\C)$ is possible. It remains to consider the $3$-dimensional
solvable, non-nilpotent Lie algebras. They are given by the Lie algebras $\Lr_{3,\la}(\C)$, and the
Lie algebra $\Lr_3(\C)$, with Lie brackets $[e_1,e_2]=e_2$ and $[e_1,e_3]=e_2+e_3$.
By proposition \ref{semisimple2} a post-Lie algebra structure on $(\Lg,\Ln)$ corresponds
to a subalgebra $\Lh$ of $\Ln \oplus \Ln$ for which the map 
$p_1-p_2: \Ln\oplus\Ln\rightarrow\Ln : (x,y)\mapsto x-y$ induces a bijection when restricted
to $\Lh$. Note that $\Lg$ is isomorphic to $\Lh$ in this case. In other words, we want
to classify the $3$-dimensional solvable, non-nilpotent Lie algebras $\Lh$ for which
there exists an injective Lie algebra homomorphism $\alpha$ such that $(p_1-p_2)\circ \alpha$ 
is bijective:
\[ 
\xymatrix{
\Lh\ar[r]^{\alpha}  \ar[dr]_{(p_1-p_2)\circ \alpha} &\Ln\oplus\Ln \ar[d]^{p_1-p_2}\\ & \Ln }
\]
Since both $p_1\circ \alpha$ and $p_2\circ \alpha$ are Lie algebra homomorphisms, their 
kernels are ideals of $\Lh$. If one of them equals $\Lh$, then the post-Lie algebra
product is either zero, or it is given by $x\cdot y=[x,y]$. In both cases $\Lh$ is isomorphic to
$\Ls\Ll_2(\C)$, as we have seen before. Suppose that one of these kernels equals zero. Then
this map is an injective Lie algebra homomorphism, so that $\Lh\simeq \Ln$ is simple. Hence,
if $\Lh$ is not simple, then  $\ker(p_1\circ \alpha)$ and  $\ker(p_2\circ \alpha)$ are both 
non-trivial ideals of $\Lh$. As a consequence we may assume that $\Lh$ is one of the Lie
algebras $\Lr_3(\C)$ or $\Lr_{3,\la}(\C)$. The non-trivial ideals of $\Lr_3(\C)$ are represented 
by $\langle e_2\rangle$ and $\langle e_2,e_3\rangle$. But then $p_1(\al(e_2))=p_2(\al(e_2))=0$,
so that $((p_1-p_2)\circ \al)(e_2)=0$. Hence  $(p_1-p_2)\circ \al$ is not bijective when
restricted to $\Lh=\Lr_3(\C)$. Hence there exists no post-Lie algebra structure on
$(\Lg,\Ln)$ with $\Ln=\Ls\Ll_2(\C)$ and $\Lg\simeq \Lr_3(\C)$. Similarly we see that there
is no post-Lie algebra structure for the unimodular Lie algebra $\Lr_{3,-1}(\C)$. This 
also follows from a general result in \cite{BD2}. \\
On the other hand it is easy to find a post-Lie algebra structure on $(\Lg,\Ln)$ for
$\Ln=\Ls\Ll_2(\C)$ and $\Lg\simeq \Lr_{3,\la}(\C)$ for all $\la\neq -1$ by direct calculation:
let $\al\neq \be$ be two complex parameters. If the brackets of $\Ln$ are given by 
\[
\{e_1,e_2\}=e_3,\;\{e_1,e_3\}=-2e_1,\; \{e_2,e_3\}=2e_2,
\]
then the following product 
$$
\begin{array}{ll}
e_2\cdot e_1 = -\al e_1+e_3, & e_3\cdot e_1 = \frac{2\be}{\al-\be}e_1, \\[0.2cm]
e_2\cdot e_2 = \al e_2+\frac{\al^2-\be^2}{4}e_3, & e_3\cdot e_2 = -\frac{2\be}{\al-\be}e_2-\be e_3, \\[0.2cm]
e_2\cdot e_3 = \frac{\be^2-\al^2}{2}e_1-2e_2, & e_3\cdot e_3 = 2\be e_1, 
\end{array}
$$
defines a post-Lie algebra structure on $(\Lg,\Ln)$, where $\Lg$ is given by the Lie brackets
\begin{align*}
[e_1,e_2] & =\al e_1,\\[0.2cm]
[e_1,e_3] & = -\frac{2\al}{\al-\be}e_1,\\[0.2cm]
[e_2,e_3] & =\frac{\be^2-\al^2}{2}e_1+\frac{2\be}{\al-\be}e_2+\be e_3.
\end{align*}
For $\be\neq 0$ and $\la=-\al/\be\neq -1$ it is isomorphic to $\Lr_{3,\la}(\C)$.
\end{proof}

\begin{rem}
As we have seen in corollary $\ref{g-sl2}$, a post-Lie algebra structure on
$(\Lg,\Ln)$ with $\Lg=\Ls\Ll_2(\C)$ only exists if  $\Ln$ is isomorphic to $\Ls\Ll_2(\C)$.
In that case we have just the two post-Lie algebra structures $x\cdot y=0$, or
$x\cdot y=[x,y]$, see proposition $\ref{gn-simple}$. In \cite{BD2} we show that if there
exists a post-Lie algebra structure on $(\Lg,\Ln)$ with $\Lg$ semisimple, then $\Ln$ cannot
be solvable. For $\dim(\Lg)=3$ this again yields corollary $\ref{g-sl2}$.
\end{rem}

\begin{prop}
Let $x\cdot y$ be a post-Lie algebra structure on $(\Lg,\Ln)$, where $\Lg$ is simple and
$\Ln$ is semisimple. Then $\Ln$ is also simple and either $x\cdot y=0$ and $[x,y]=\{x,y\}$, 
or $x\cdot y=[x,y]=-\{x,y\}$. 
\end{prop}

\begin{proof}
By lemma $\ref{2.5}$ we know that the map $L\colon \Lg \ra \Der(\Ln)=\ad(\Ln)\simeq \Ln$, 
$x\mapsto L(x)$ is a Lie algebra homomorphism. Its kernel is an ideal in $\Lg$. 
If $L$ is the zero map, then $x\cdot y=0$ for all $x,y\in \Lg$. Otherwise $L$ is a monomorphism and 
$\Lg$ embedds into $\Ln$, so that $\Ln\simeq \Lg$ is also simple. The claim follows from 
proposition $\ref{gn-simple}$.
\end{proof}

\end{document}